\documentclass{article}
\usepackage{amsmath,amssymb}
\usepackage{amsthm}
\usepackage{graphicx}
\usepackage{tikz}
\usepackage[subrefformat=parens]{subcaption}
\usepackage[usenames,dvipsnames]{pstricks}
\usepackage{epsfig}
\usepackage{float}
\usepackage{authblk}
\usepackage[colorlinks=true,citecolor=black,linkcolor=black,urlcolor=blue]{hyperref}

\newtheorem{theorem}{Theorem}[section]

\newtheorem{corollary}{Corollary}[section]

\title{A simple proof for the chromatic number \\of cyclic Latin squares of even order}
\author{Zahra Naghdabadi }
\affil{\small Department of Mathematical Sciences, \\ \small Sharif University of Technology, Tehran, Iran.}
\date{} 

\begin{document}
	
	\maketitle
	
	\begin{abstract}
		The chromatic number of a cyclic Latin square of order $2n$ is $2n+2$. The available proof for this statement includes a coloring that is rather lengthy. Here, we introduce a coloring of cyclic Latin squares of even order $2n$ (the Latin square graph of a cyclic group's Cayley table) with $2n+2$ colors using a simple method supported by a graphical presentation.\\
		
		
	\end{abstract}

	\section{Introduction}

Let $T_{n}$ be the Cayley table of the cyclic group of order $n$. Then, $T_{n}$ is a Latin square with the following pattern (Figure \ref{fig:3a}). The $(i,j)$th entry of $T_{n}$ has label $i+j-2 \ \left( \text{mod} \ n\right) $. Here, $T_{n}$ is called a cyclic Latin square.

We construct the Latin square graph of $T_{n}$ as follows. Let $G=(V,E)$ be a graph whose vertices are cells of $T_{n}$ and two vertices are adjacent if they are in the same row, in the same column, or have the same labels.\\
$$V=\{(i,j)| 1 \leq i,j \leq n \}$$
$$ \, \, \{(i,j),(a,b)\} \in E \Longleftrightarrow i=a \ \text{or} \ j=b\ \text{or} \ i+j=a+b \ \left( \text{mod} \ n\right)  $$
Then, $\chi (T_{n})$ denotes the chromatic number of the cyclic Latin square graph obtained from $T_{n}$. The chromatic number of a graph is the minimum number of colors required to color the vertices of the graph in such a way that adjacent vertices have different colors.

The chromatic number of cyclic Latin square graphs was found in 2016 \cite{2}. Some bounds were found for chromatic number of general Latin square graphs in \cite{4}. Later in \cite{3}, these bounds were improved and the case for Cayley tables of abelian groups was completely solved.

However, in \cite{2} for proof of the chromatic number of cyclic Latin squares of even order, a coloring was introduced which was rather complicated. Here, we give a simple coloring of the cyclic Latin square of order $2n$ by $2n+2$ colors. 
	\section{Chromatic number of cyclic Latin squares}
	 The following Theorem for chromatic number of the Latin square graph of $T_{n}$ was proved in \cite{2} by a long argument for the even case.
	
	\begin{theorem} \label{th}
		$\chi(T_{n}) = n$ if $n$ is odd and $\chi(T_{n}) = n+2$ if $n$ is even. 
	\end{theorem}
	
	\begin{proof}
		\begin{itemize}
		\item \emph{Odd Case:} Assume $n$ is \textbf{odd}. Clearly, $\chi(T_{n}) \geq n$.
		To find an $n$-coloring, we choose $n$ colors for cells of the first row. Then we shift colors for next rows and we are done.
		We choose colors $\{1,2,\ldots ,n\}$, then the color of the $(i,j)$th entry is $j-i \ (\text{mod} \ n)$ (Fig \ref{fig:3}). It is straight forward to see that this is indeed a coloring.
		
		\begin{figure}[H] \label{fig 3}
			\centering
			\begin{subfigure}{0.22\textwidth}
				\includegraphics[width=\linewidth]{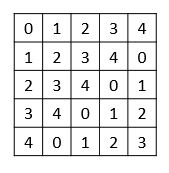}
				\caption{$T_{5}$, Cyclic Latin square of order 5} \label{fig:3a}
			\end{subfigure}
			\hspace{0.035\textwidth} 
			\begin{subfigure}{0.22\textwidth}
				\includegraphics[width=\linewidth]{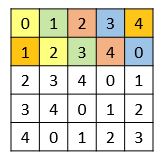}
				\caption{Coloring of $T_{5}$ by shifting colors} \label{fig:3b}
			\end{subfigure}
			\hspace{0.035\textwidth} 
			\begin{subfigure}{0.22\textwidth}
				\includegraphics[width=\linewidth]{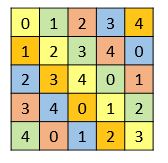}
				\caption{5-coloring of $T_{5}$} \label{fig:3c}
			\end{subfigure}
			\caption{Example of a cyclic Latin square of odd order $n$ and its $n$-coloring.} \label{fig:3}
		\end{figure} 
		
		\item \emph{Even Case:} Assume that $n$ is \textbf{even}. By the following Corollary $\chi(T_{n}) \geq n+2$.
		\begin{corollary}
			Let $G$ be an abelian group of order $n$. Then, $\chi(T_{G}) \geq n + 2$ if and only if $G$ has a unique element of order 2. Here, $T_{G}$ is the Latin square graph obtained from the Cayley table of $G$.
		\end{corollary}	
		This corollary was deduced from several papers in \cite{2}.
		
		\begin{itemize}
		\item Here, we introduce a \textbf{coloring} of $T_{n}$ with $n+2$ colors to show that $\chi(T_{n})= n+2$. First, we add two empty columns to the Latin square (Fig \ref{fig:4}). We choose $n+2$ colors for the first row and shift the colors for next rows, consecutively. The colors are shifted twice only for the $(\frac{n}{2}+1)$th row (Fig \ref{fig:4c}). We choose colors $\{1,2,\ldots ,n+2\}$, then the color of the $(i,j)$th cell is $j-i \ \left( \text{mod} \ n+2\right)$ in the first half $i \leq \frac{n}{2}$ and $j-i-1 \ \left( \text{mod} \ n+2\right)$ in the second half $i > \frac{n}{2}$.		
		\begin{figure}[H] \label{fig 4}
			\centering
			\begin{subfigure}{0.25\textwidth}
				\includegraphics[width=\linewidth]{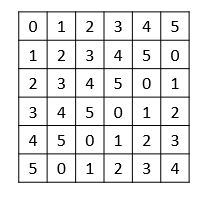}
				\caption{$T_{6}$, Cyclic Latin square of order 6} \label{fig:4a}
			\end{subfigure}
			\hspace*{\fill} 
			\begin{subfigure}{0.34\textwidth}
				\includegraphics[width=\linewidth]{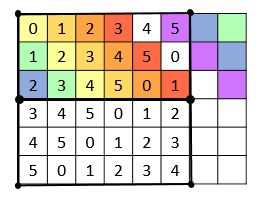}
				\caption{Coloring of the first half of $T_{6}$ by shifting colors} \label{fig:4b}
			\end{subfigure}
			\hspace*{\fill} 
			\begin{subfigure}{0.34\textwidth}
				\includegraphics[width=\linewidth]{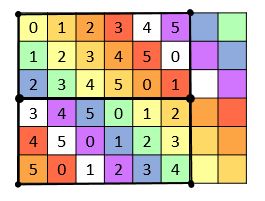}
				\caption{8-coloring of $T_{6}$} \label{fig:4c}
			\end{subfigure}
			\caption{Example of a cyclic Latin square of even order $n$ and its $(n+2)$-coloring.} \label{fig:4}
		\end{figure} 
				
		\item \textbf{This is indeed a coloring.} We need to show that color of the $(i,j)$th cell is different from its adjacent vertices. By construction, if two cells are in the same row, they have different colors. The same is true for vertices in the same column, since $T_{n}$ could be embedded in a square of dimension $n+2$ colored by simple shifts. Now we see that if two cells have the same color, they have different labels and the proof is finished.
		
		Fix a color $C$. We note that labels of color $C$ in each half are either all even or all odd since the cyclic Latin square has even order. If cells of color $C$ have even labels in the first half, they have odd labels in the second half and vice versa. So, two cells of color $C$ have different labels if one is in the first half and the other is in the second half. 
		
		To see that two cells of color $C$ in the same half have different colors, we use the cyclic structure of the Latin square (Fig's \ref{fig:5e} and \ref{fig:5f}). Without loss of generality, assume that cells with color $C$ have even labels in the first half. Then, starting from the first row, labels of these cells form an ordered sequence of length less than or equal to $\frac{n}{2}$ satisfying the order in Figure \ref{fig:5c}. Note that in some rows colored cells do not have labels and that after two empty labels in the sequence, entries continue from the second label in the order as illustrated in Figure \ref{fig:5e}. This order results from the fact that in each half labels of color $C$ differ by 2 in two consecutive rows and by 4 when there are two blank cells separating them in the sequence (Fig \ref{fig:5f}).
		 Since the ordered sequence is chosen in order from $\frac{n}{2}$ possible labels, it has distinct entries. Hence, labels of color $C$ are different in the first half and they are all even. The same is true for labels in the second half. They are odd and different.	
		 	\end{itemize}
		\end{itemize}	
\end{proof}
\begin{figure} \label{fig 5}
	\centering
\begin{subfigure}{0.45\textwidth}
	\includegraphics[width=\linewidth]{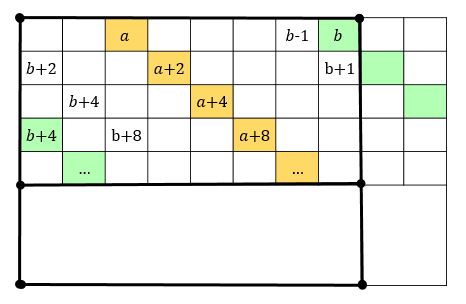}
	\caption{Labels of the same color differ by 2 in two consecutive rows and by 4 when there are two blank cells separating them in the sequence.} \label{fig:5e}
\end{subfigure}
\hspace{0.04\textwidth}
\begin{subfigure}{0.27\textwidth}
	\includegraphics[width=\linewidth]{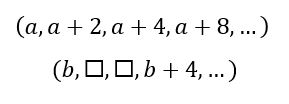}
	\caption{The sequence of labels in each half has length $\leq \frac{n}{2}$ because some colored cells are empty.} \label{fig:5f}
\end{subfigure}

	\begin{subfigure}{0.27\textwidth}
		\includegraphics[width=\linewidth]{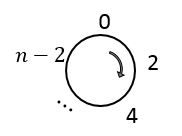}
		\caption{Order of even labels appearing in sequence of labels} \label{fig:5c}
	\end{subfigure}
	\hspace{0.04\textwidth}
	\begin{subfigure}{0.27\textwidth}
		\includegraphics[width=\linewidth]{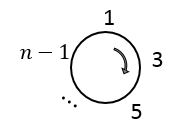}
		\caption{Order of odd labels appearing in sequence of labels} \label{fig:5d}
	\end{subfigure}
	\caption{Graphical proof of Theorem \ref{th}: Cells with the same color have different labels since even and odd labels appear in particular orders differing by 2.} \label{fig:5}
\end{figure} 
In the coloring stated for the proof of even cases, four colors, $(\frac{n}{2}, \frac{n}{2}+1,n+1, n+2)$ appear $n-1$ times in the coloring. These colors appear on entries $(\frac{n}{2}+1,n+2), (\frac{n}{2},n+1), (1,n+2)$ and $(n,n+1)$, respectively in the extra columns. The other colors appear exactly $n-2$ times. Hence, the coloring is \textit{equitable} (the cardinality of any two color classes differ by at most 1) as defined in \cite{4}. 

\section*{Acknowledgment}
The author is grateful to Professor E.S. Mahmoodian for his comments and reviewing the article. The author would also like to thank the anonymous reviewer whose detailed comments helped to improve the quality of the paper.

\end{document}